\documentclass[12pt]{amsart}

\usepackage{amsthm,amssymb,amsmath}

\newtheorem{T}{Theorem} 
\newtheorem{Le}{Lemma} 
\newtheorem{C}{Corollary}

\theoremstyle{definition}

\newtheorem{E}{Example}

\theoremstyle{remark}
\newtheorem{R}{Remark}

\begin{document}

\date{\today}

\title{Circle Decompositions of Surfaces} 

\author{G\'abor Moussong and N\'andor Sim\'anyi}

\thanks{The first author was supported by the Hungarian Nat.\ Sci.\ Found.\ 
  (OTKA) grant No.\ T047102, the second author was supported by the National Science
  Foundation, grant DMS-0800538.}

\address[G\'abor Moussong]{Inst.\ of Mathematics\\ E\"otv\"os
  University\\ P\'azm\'any P\'eter S\'et\'any 1/C\\
  Budapest 1117, Hungary}
\email{mg@math.elte.hu}
\address[N\'andor Sim\'anyi]{The University of Alabama at 
  Birmingham\\ Department of Mathematics\\
  1300 University Blvd., Suite 452\\
  Birmingham, AL 35294 U.S.A.}
\email{simanyi@math.uab.edu}


\begin{abstract} 
We determine which connected surfaces can be partitioned into
topological circles.  There are exactly seven such surfaces up to
homeomorphism: those of finite type, of Euler characteristic zero, and
with compact boundary components.  As a byproduct, we get that any
circle decomposition of a surface is upper semicontinuous.
\end{abstract}

\maketitle 

\parskip=0pt plus 2.5pt

\section{Introduction}
\label{sec_intr}

In what follows by a surface we shall mean a second countable, Hausdorff, connected, two-dimensional, topological manifold, possibly with boundary. By a circle in a surface $S$ we shall mean a closed Jordan curve, i.e., any subset of $S$ homeomorphic to the standard unit circle. A circle decomposition of $S$ is a partition of $S$ into circles.

Our goal is to show that circle decompositions only exist for a very limited range of surfaces. The main result in this note is Corollary~\ref{allowable} below stating that any surface with a circle decomposition is homeomorphic to either a torus, a Klein bottle, an annulus, a M\"obius band, an open annulus, a half-open annulus, or an open M\"obius band. For short, these seven topological types will be called allowable surfaces. 

One may observe that these allowable surfaces are precisely those of finite type (i.e., with finitely generated homology), with zero Euler characteristic, and with all boundary components homeomorphic to a circle.

It is clear by straightforward geometric constructions that all allowable surfaces admit circle decompositions. Moreover, such constructions can be carried out in the smooth category resulting in smooth foliations with circles. It is well known that any surface foliated by circles has zero Euler characteristic, therefore, circle foliations can only exist for allowable surfaces. However, as the following example shows, a circle decomposition need not be a topological foliation. The authors are indebted to Lex Oversteegen for pointing out the existence of such examples.

\begin{E}\label{non-foliation} Construct first of all a smooth, centrally symmetric partition $\mathcal J$ of the square $Q=[-1,1]\times [-1,1]$ into Jordan arcs $J$ such that

\begin{enumerate}
  \item $\{ -1\}\times [-1,1]\in\mathcal J$ and $\{ 1\}\times [-1,1]\in\mathcal J$,
  \item every $J\in\mathcal J$, other than the two curves listed in (1), has only its endpoints on the boundary $\partial Q$, one on the bottom, one on the top side of $Q$, and
  \item the element of $\mathcal J$ that contains the origin is
$$J^{\ast} =\left\{ (x,y)\in Q:-\frac{1}{2}\le x\le\frac{1}{2},\; y=12x^3-x\,\right\}.$$
\end{enumerate}

Next we horizontally shrink by a factor of $1/2$ the square $Q$ along with its smooth partition $\mathcal J$, and insert the arising block in the left half $[-1,0]\times [-1,1]$ of $Q$. Then we horizontally shrink the square $Q$, along with its smooth partition $\mathcal J$, by a factor of $1/4$, and insert the arising block in the rectangle $[0,1/2]\times [-1,1]$. After this we again horizontally shrink the square $Q$, along with its smooth partition $\mathcal J$, by a factor of $1/8$, and insert the arising block in the rectangle $[1/2,3/4]\times [-1,1]$, etc.

Finally, we include the right-hand edge of $Q$ to complete a partition $\mathcal P$ of the square $Q$ into Jordan arcs. It is clear that the partition elements of $\mathcal P$ can be parametrized as
$$\mathcal P=\{J_t:0\le t\le 1\}$$
in accordance with the horizontal linear order among the curves $J_t\in\mathcal P$. This parametrization is continuous with respect to the Hausdorff metric on the set of compacta in $Q$. Thus, the projection map, which takes all elements of $J_t$ to $t$, is an open quotient map from $Q$ to $[0,1]$.

Yet the partition fails to be topologically equivalent to the canonical partition of $Q$ into vertical line segments. Indeed, the partition $\mathcal P$ does not even possess any transversal curve emanating from the point $(1,y)\in Q$ if $-1/9<y<1/9$ ($\pm 1/9$ being the local maximum and minimum of $y$-values along $J^{\ast}$).\hfill\qed
\end{E}

Squares partitioned as in this example can clearly be involved in circle decompositions of surfaces. So, circle decompositions in general are not foliations. There is however a weaker property of decompositions, namely, upper semicontinuity (see Section~\ref{sec_usc}) which, as we show in Corollary~\ref{usc} and the subsequent remark, is shared by all circle decompositions of surfaces. By the classical Jordan--Sch\"onflies theorem, circles in surfaces have strong local separaton properties, which forces circle decompositions to be upper semicontinuous. Our main theorem will follow relatively easily from this fact in Section~\ref{sec_proof}.

It should be noted that many circle decompositions of $3$-manifolds exist (for instance, Euclidean $3$-space can be foliated by circles, see \cite{V}) which in general are not upper semicontinuous. 
    
\section{Preliminary lemmas}

\begin{Le}\label{nodisk}
  A closed disk admits no circle decomposition.
\end{Le}

\begin{proof} Suppose that $\mathcal C$ is a circle decomposition of the closed disk $S$. Any member $C\in\mathcal C$ has a well-defined interior; namely, the connected component of $S-C$ which does not contain $\partial S$. For $C_1,C_2\in\mathcal C$ call $C_1<C_2$ if $C_1$ is contained in the interior of $C_2$. One readily checks that $<$ is a partial order relation on the set $\mathcal C$. Compactness of $S$ implies that any ordered chain in $\mathcal C$ has a lower bound. Then by Zorn's lemma there exists at least one minimal element in $\mathcal C$. But no minimal elements can exist since the interior of any circle must contain further circles.
\end{proof}

\begin{C}\label{circle_bounds_no_disk}
  If $\mathcal C$ is a circle decomposition of the surface $S$, then no element of $\mathcal C$ bounds a disk in $S$.\hfill\qed
\end{C}

It follows for instance that neither an open disk nor a two-sphere admit circle decompositions. 

\begin{Le}\label{circle_in_annulus}
Let $C_0$ and $C_1$ be the two boundary circles of an annulus $A$. If $D\subseteq A$ is a circle with $D\ne C_0$, then there exists at least one connected component $U$ of $A-D$ with $U\cap C_1=\emptyset$. If, further, $D\cap C_0\ne\emptyset$, then any circle in such a component $U$ bounds a disk in $A$.
\end{Le}

\begin{proof}
The set $A-D$ is disconnected unless $D=C_1$. Therefore, the connected set $C_1$ cannot intersect all connected components of $A-D$. This implies the first statement. Let $U$ be a connected component of $A-D$ with $U\cap C_1=\emptyset$. Suppose now that $C\subseteq U$ is a circle which is not nullhomotopic in $A$. Then $C$ and $C_1$ bound an annulus which contains $D$. Therefore, both $D$ and $C_1$ lie on the same side of $C$ in $A$, which implies that $D$ cannot intersect $C_0$. This proves the last statement.
\end{proof}

\begin{Le}\label{annulus_nice}
Suppose that $\mathcal C$ is a circle decomposition of an annulus $A$. Then both boundary circles of $A$ belong to $\mathcal C$. If $C\in\mathcal C$ is contained in the interior of $A$, then $C$ cuts $A$ into two annuli (both of which inherit circle decompositions from $\mathcal C$).
\end{Le}

\begin{proof}
Assume that $C\in\mathcal C$ is different from both boundary circles. If $C$ is not contained entirely in the interior of $A$, then Lemma~\ref{circle_in_annulus} applied to $D=C$ implies that some elements of $\mathcal C$ bound disks, contradicting Corollary~\ref{circle_bounds_no_disk}. This proves the first statement of the lemma. By Corollary~\ref{circle_bounds_no_disk} only homotopically nontrivial circles can belong to~$\mathcal C$, which implies the last statement.
\end{proof}

The following lemma reduces the main question to the case when the surface is orientable and has no boundary. Recall from the introduction that we call a surface allowable if it is homeomorphic to one from the following list: torus, Klein bottle, annulus, M\"obius band, open annulus, half-open annulus, open M\"obius band.

\begin{Le}\label{reduction}
Suppose that the surface $S$ is not allowable, and that $S$ admits a circle decomposition. Then there exists an orientable surface $\widetilde{S}$ with  $\partial\widetilde{S}=\emptyset$ which is not allowable and admits a circle decomposition.
\end{Le}

\begin{proof}
First we eliminate boundary by attaching parts with circle decompositions to each boundary component of $S$. Let $U$ denote a half-open annulus with a circle decomposition of its interior, and glue a copy of $U$ along $\partial U$ to each compact component of $\partial S$. Let $V$ be a closed half-plane with an interior point removed, and equip the interior of $V$ (an open annulus) with a circle decomposition. Glue a copy of $V$ along its boundary to each noncompact component of $\partial S$. The surface $S'$ obtained through all these gluings has $\partial S'=\emptyset$, and inherits a circle decomposition from $S$ and from the attached parts.

Next, if $S'$ is orientable, put $\widetilde{S}=S'$, if not, then define $\widetilde{S}$ as the orientable double covering of $S'$. The circle decomposition of $S'$ clearly lifts to that of $\widetilde{S}$. 

It remains to be shown that $\widetilde{S}$ is not allowable. By inspection of the list of allowable surfaces it is clear that a surface is allowable if and only if any double covering of it is allowable. Therefore, it suffices to check that $S'$ is not allowable. 

We know that $S$ is not allowable. If $S$ is not of finite type, then neither is $S'$. So suppose that $S$ has finite type. If all boundary components of $S$ are compact, then $S'$ is homotopy equivalent to $S$, so it is not allowable. Suppose now that $S$ has at least one boundary component homeomorphic to the real line. Then $\pi_1(S')$ is the free product of $\pi_1(S)$ with at least one copy of $\mathbf Z$, therefore, it can only be isomorphic to the fundamental group of one of the allowable surfaces if $S$ is simply connected. But then $S$ cannot have a circle decomposition by Corollary~\ref{circle_bounds_no_disk}. So $S'$ is not allowable in this case either.
\end{proof}

\section{Upper semicontinuity}
\label{sec_usc}

Suppose that a Hausdorff topological space $X$ is decomposed to a family $\mathcal F$ of pairwise disjoint compact sets. Recall that $\mathcal F$ is an upper semicontinuous decomposition if for every $F\in\mathcal F$ and for every neighborhood $U$ of $F$ there exists a smaller neighborhood $V$ of $F$ such that $G\subseteq U$ whenever $G\in\mathcal F$ and $G\cap V\ne\emptyset$.

\medskip

We shall prove that all circle decompositions of surfaces are upper semicontinuous, see Corollary~\ref{usc} below.  First we use embedded annuli to characterize upper semicontinuous circle decompositions of orientable surfaces with empty boundary. Let $S$ be a surface (not necessarily orientable), and $C$ be a circle in $S$. Assume that either

\begin{enumerate}
 \item $C$ is contained in the interior of $S$, and is two-sided in $S$, or 
 \item $C$ is a component of $\partial S$.
\end{enumerate}

\noindent By an annular neighborhood of $C$ we mean any annulus embedded in $S$ for which

\begin{enumerate}
 \item $C$ is the image of the middle circle of the annulus under the embedding in the first case, or
 \item $C$ is one of the two boundary circles in the second case, respectively.
\end{enumerate}

\noindent It follows from the classical Jordan--Sch\"onflies theorems that annular neighborhoods exist for $C$, and, consequently, they form a basis of neighborhoods for $C$ in $S$.

\medskip

Suppose now that a circle decomposition $\mathcal C$ is given on $S$. An embedded annulus in $S$ will be called a $\mathcal C$-annulus if both boundary circles belong to $\mathcal C$. Accordingly, annular neighborhoods of circles will be called $\mathcal C$-annular neighborhoods if they are $\mathcal C$-annuli.

\begin{Le}\label{annulus_usc}
  Any circle decomposition of an annulus is upper semicontinuous.
\end{Le}

\begin{proof}
Let $C_0$ and $C_1$ denote the boundary circles of $A$. By Lemma~\ref{annulus_nice} both belong to $\mathcal C$. For any two further $C,C'\in\mathcal C$ write $C<C'$ if $C$ separates $C_0$ from $C'$ (or, equivalently, if $C'$ separates $C$ from $C_1$). Extend relation $<$ to $C_0$ and $C_1$ by making them smallest and largest, respectively. Then Lemma~\ref{annulus_nice} implies that the set $\mathcal C$ is linearly ordered by relation $<\,$, and that this ordering is dense. 

It follows now that any $C\in\mathcal C$ equals the intersection of its $\mathcal C$-annular neighborhoods. Indeed, if $x$ is an arbitrary point of $A$ not in $C$, then $x\in C'$ with some $C'\in\mathcal C$ for which we may assume $C<C'$. Pick $C''\in\mathcal C$ with $C<C''<C'$, then $C_0$ and $C''$ bound a $\mathcal C$-annular neighborhood of $C$ not containing $x$.

The family of $\mathcal C$-annular neighborhoods of $C$ is closed under finite intersections. Therefore, by compactness, any neighborhood of $C$ contains a $\mathcal C$-annular neighborhood. Now upper semicontinuity follows immediately.
\end{proof}

\begin{R}
It is well known that upper semicontinuity is equivalent to the Hausdorff property of the quotient space.  It is easy to see that under the assumptions of Lemma~\ref{annulus_usc} the quotient space is homeomorphic to $[0,1]$.
\end{R}

\begin{Le}\label{usc=ann_nbhds}
Let $S$ be an orientable surface without boundary, and let $\mathcal C$ be a circle decomposition of $S$. Then the following two conditions are equivalent:
\begin{enumerate}
 \item $\mathcal C$ is upper semicontinuous.
 \item Every $C\in\mathcal C$ admits a $\mathcal C$-annular neighborhood.
\end{enumerate}
\end{Le}

\begin{proof}
(1)$\Rightarrow$(2): If $\mathcal C$ is assumed upper semicontinuous and $C\in\mathcal C$ is given, choose a neighborhood $V$ of $C$ such that all members of $\mathcal C$ that meet $V$ stay in the interior of a fixed annular neighborhood $A$ of $C$. Pick two such members $C_1$ and $C_2$ on either side of $C$. By Lemma~\ref{annulus_nice} $C_1$ and $C_2$ bound an annulus within $A$, which therefore is a $\mathcal C$-annular neighborhood of $C$.

\noindent (2)$\Rightarrow$(1): Immediate consequence of Lemma~\ref{annulus_usc}.
\end{proof}

\begin{T}\label{main}
Let $\mathcal C$ be a circle decomposition of the surface $S$, and let $C_0$ be a circle in $S$. Assume that either

\begin{enumerate}
 \item $C_0$ is a component of $\partial S$, or 
 \item $C_0\in\mathcal C$ and $C_0$ is a two-sided circle in the interior of $S$. 
\end{enumerate}

\noindent Then $C_0$ has a $\mathcal C$-annular neighborhood. In particular, $C_0$ belongs to $\mathcal C$ in both cases.
\end{T}

\begin{proof}
It will suffice to prove the theorem in case (1). Indeed, the other case follows if we cut $S$ along $C_0$, find $\mathcal C$-annular neighborhoods on both sides and reglue them.

Fix an arbitrary annular neighborhood $A$ for $C_0$ in $S$. Then $C_0$ is one of the boundary circles of the annulus $A$. We shall prove that there exists a circle $D\in\mathcal C$ different from $C_0$, and contained in $A$. If such a $D$ is found, then Lemma~\ref{circle_in_annulus} and Corollary~\ref{circle_bounds_no_disk} imply that $C_0\cap D=\emptyset$, and that $D$ is not nullhomotopic in $A$. Then $C_0$ and $D$ bound an annulus which inherits a circle decomposition from $\mathcal C$. Lemma~\ref{annulus_nice} applied to this annulus yields $C_0\in\mathcal C$. So, $C_0$ and $D$ bound a $\mathcal C$-annular neighborhood for $C_0$ in $S$.

By way of contradiction, for the rest of the proof we assume that no circle $D\in\mathcal C$ exists with $D\ne C_0$ and $D\subseteq A$. Now we introduce some further notation. For concreteness, let us fix a homeomorphism $h:C_0\times [0,1]\to A$ with $h(x,0)=x$ for $x\in C_0$. For any parameter $t\in (0,1]$ define the following sets:

\begin{displaymath}
  \begin{aligned}
    C^t= & h(C_0\times\{t\}), \\
    A^t= & h(C_0\times [0,t]), \\
    U^t= & A^t-C^t=h(C_0\times [0,t)), \\
    \mathcal D^t= &\{D\in\mathcal C: D\ne C_0\;\text{and}\;D\cap U^t\ne\emptyset\}.
  \end{aligned}
\end{displaymath}

\noindent Then $A^t$ is an annulus bounded by $C^0=C_0$ and $C^t$; in particular, $A^1=A$. (We shall only use these sets for $t=1$, $t=1/2$, and $t=1/3$, that is, for the annulus $A$ and for its half and third.)

By our assumption for any $D\in\mathcal D^t$ the set $D\cap U^t$ is a disjoint union of open arcs in $D$. Let the closures of all such arcs (with fixed $t$ and variable $D$) form the set $\mathcal E^t$. 

Any element $E$ of $\mathcal E^t$ is a Jordan arc in $A^t$ connecting two distinct points of $C^t$. One side of this arc in the half-open annulus $U^t$ is an open disk $V_E$. Let $K_E$ denote the closure of $U^t-V_E$ in $A^t$, then $K_E$ is compact and connected.

Two distinct elements $E_1,E_2\in\mathcal E^t$ cannot intersect one another in $U^t$. This implies that $V_{E_1}$ and $V_{E_2}$ are either disjoint, or one is contained in the other. Moreover, if $E_1\ne E_2 $ and $V_{E_1}\subseteq V_{E_2}$, then $E_1\cap U^t\subseteq V_{E_2}$. Therefore, for any finite number of elements $E_1$, $\ldots$, $E_k\in\mathcal E^t$ the set $K_{E_1}\cap\ldots\cap K_{E_k}$ is still connected. If in a family of continua all finite subfamilies have connected intersection, then the intersection of the whole family is a continuum. This implies that the set $K^t=\bigcap\{K_E:E\in\mathcal E^t\}$ is connected. Our goal is to show that $K^t=C_0$. Clearly $K^{t_1}\subseteq K^{t_2}$ whenever $t_1\le t_2$. The relation $K^t=C_0$ is actually true for all $t$, but for our purposes it will suffice to prove this for one particular value $t<1$. For concreteness, let us select $t=1/2$ and put $K=K^{1/2}$.

We claim now that $K=C_0$. To this end consider first the family $\mathcal F=\{E\cap K: E\in\mathcal E^1, E\cap K\ne\emptyset\}$ of compact subsets of the circles in the decomposition $\mathcal C$. Since no two distinct arcs in $\mathcal E^1$ can intersect in $U^1$, all elements of $\mathcal F$ are pairwise disjoint. Each member $F$ of $\mathcal F$ is contained in a unique arc $E(F)\in\mathcal E^1$. The correspondence $F\mapsto E(F)$ is clearly injective.

Consider the sets of the form $V_{E(F)}$ for $F\in\mathcal F$; these are all nonempty open sets in $S$. We claim that they are pairwise disjoint. Indeed, if $V_{E(F_1)}$ and $V_{E(F_2)}$ intersect, then one is a subset of the other, say, $V_{E(F_1)}\subseteq V_{E(F_2)}$. Now if $F_1\ne F_2$, then $E(F_1)\ne E(F_2)$, and by our previous arguments $E(F_1)\cap U^1\subseteq V_{E(F_2)}$. But this is impossible since $E(F_1)\cap K\ne\emptyset$ while $V_{E(F_2)}$ is disjoint from $K^1$ and $K^1\supseteq K$.

It follows that $\mathcal F$ is countable. Now if $C_0\in\mathcal C$, then $K=C_0\cup\bigcup\mathcal F$, and if $C_0\notin\mathcal C$, then $K=\bigcup\mathcal F$. In both cases the continuum $K$ is decomposed into a countable family of pairwise disjoint closed subsets. By a theorem of Sierpi\'nski (\cite{S}) this is only possible if the family consists of a single set. Clearly $K\notin\mathcal F$ since $C_0\subseteq K$. Therefore, only the case $C_0\in\mathcal C$ and $K=C_0$ is possible, and our claim is proved. 

\smallskip
Finally, consider the parallel circle $C^{1/3}$ of the annulus $A$. Since it is disjoint from $K$, the set $C^{1/3}$ is covered by the family of open disks $V_E$ for $E\in\mathcal E^{1/2}$. By compactness a finite number of these disks cover $C^{1/3}$. If two of these disks is not disjoint, then one is contained in the other, therefore a minimal such covering can only consist of a single set $V_E$. This is clearly impossible if $E$ is a Jordan arc connecting two points of $C^{1/2}$ in $A^{1/2}$. This contradiction proves the theorem.
\end{proof}

\begin{C}\label{usc}
If $\mathcal C$ is a circle decomposition of a surface with empty boundary, then $\mathcal C$ is upper semicontinuous.
\end{C}

\begin{proof}
If the surface is orientable, then Lemma~\ref{usc=ann_nbhds} combined with Theorem~\ref{main} gives the result. In the non-orientable case the circle decomposition lifts to a circle decomposition of the orientable double covering. Upper semicontinuity of the latter obviously implies upper semicontinuity of the former.
\end{proof}

\begin{R}
It is also true that all circle decompositions of surfaces are upper semicontinuous, that is, in Corollary~\ref{usc} the surface $S$ may have boundary. If all connected components of $\partial S$ are circles, then this follows from Theorem~\ref{main}. One may prove directly that if $S$ has a circle decomposition, then none of the boundary components can be homeomorphic to the real line. We omit this proof since this fact will follow from Corollary~\ref{allowable}.
\end{R}

\section{Proof of the main theorem}
\label{sec_proof}

\begin{T}\label{torus_annulus}
Let $S$ be an orientable surface without boundary. If there exists a circle decomposition of $S$, then $S$ is either a torus or an open annulus.
\end{T}

\begin{proof}
Let $\mathcal C$ be a circle decomposition of $S$. Call two members of $\mathcal C$ equivalent if they are equal or bound a $\mathcal C$-annulus in $S$. This is clearly an equivalence relation, and Theorem~\ref{main} implies that the union of each equivalence class is open in $S$. Since $S$ is connected, there is a single class.

Observe that if two $\mathcal C$-annuli in $S$ are not disjoint, then their union is either again a $\mathcal C$-annulus, or else is a torus which equals $S$.

If $K\subseteq S$ is any connected compact set, then repeated application of this last observation shows that either $S$ is a torus, or $K$ is covered by a single $\mathcal C$-annulus.

So, if $S$ is compact, then it is a torus. If $S$ is not compact, then one can exhaust $S$ by an increasing sequence of connected compact subsets, therefore, $S$ can be exhausted by a strictly increasing sequence of $\mathcal C$-annuli. The union of such a sequence is an open annulus, so in the noncompact case $S$ is an open annulus.
\end{proof}

\begin{C}\label{allowable}
  If a surface $S$ admits a circle decomposition, then $S$ is allowable.
\end{C}

\begin{proof}
If $S$ were not allowable, then Lemma~\ref{reduction} would produce $\widetilde{S}$, orientable without boundary, still not allowable, and still admitting a circle decomposition. But Theorem~\ref{torus_annulus} implies that such an $\widetilde{S}$ must be a torus or an open annulus, both of which are allowable, a contradiction.
\end{proof}

\end{document}